\documentclass[a4paper,11pt,oneside]{article}
\usepackage[margin=1in]{geometry}
\pdfoutput=1

\usepackage{amsmath, amsthm, amssymb}
\usepackage{hyperref}
\usepackage{libertine}
\usepackage{courier}
\usepackage{mathpazo}
\usepackage{mathtools}
\usepackage{natbib}

\hypersetup{
	colorlinks=true,
  linkcolor=red,          %
  citecolor=blue,         %
  filecolor=magenta,      %
  urlcolor=cyan           %
}

\newtheorem{theorem}{Theorem}[section]
\newtheorem{lemma}{Lemma}[section]

\newtheorem{proposition}{Proposition}[section]

\newtheorem{corollary}{Corollary}[section]
\newtheorem{remark}{Remark}[section]

\newcommand{\set}[1]{\left\{ #1 \right\}}

\newcommand{\pred}[1]{\boldsymbol{1}[#1]}

\newcommand{\R}{\mathbb{R}}
\newcommand{\C}{\mathbb{C}}

\newcommand{\N}{\mathbb{N}}

\newcommand{\zero}{\boldsymbol{0}}

\newcommand{\expo}[1]{\exp\left( #1 \right)}

\DeclareMathOperator{\Bin}{Bin}

\newcommand{\abs}[1]{\left| #1 \right|}
\newcommand{\paren}[1]{\left( #1 \right)}
\newcommand{\nrm}[1]{\left\Vert #1 \right\Vert}
\newcommand{\hfnrm}[1]{\nrm{#1}_{1/2}}
\newcommand{\vertiii}[1]{{\left\vert\kern-0.25ex\left\vert\kern-0.25ex\left\vert #1 
    \right\vert\kern-0.25ex\right\vert\kern-0.25ex\right\vert}}
\newcommand{\tv}[1]{\nrm{#1}_{\textup{\tiny\textsf{TV}}}}

\DeclarePairedDelimiter\floor{\lfloor}{\rfloor}

\newcommand{\X}{\boldsymbol{X}}
\newcommand{\PR}[2][]{\mathop{\mathbb{P}}_{#1}\left( #2 \right)}
\newcommand{\E}[2][]{\mathop{\mathbb{E}}_{#1}\left[ #2 \right]}

\newcommand{\dist}{\bmu}

\newcommand{\empdist}{\widehat{\dist}_m}
\newcommand{\altest}{\tilde\bmu_m}
\newcommand{\disttr}[1]{\dist[#1]}
\newcommand{\empdisttr}{\empdist'}

\newcommand{\calE}{\mathcal{E}}

\newcommand{\eps}{\varepsilon}

\DeclareMathOperator{\Unif}{Uniform}

\newcommand{\rad}{\mathfrak{R}}
\newcommand{\emprad}{\hat{\mathfrak{R}}}
\newcommand{\risk}{\mathcal{R}}
\newcommand{\riskf}{\bar{\mathcal{R}}}

\newcommand{\bmu}{\boldsymbol{\mu}}
\newcommand{\bnu}{\boldsymbol{\nu}}

\newcommand{\bigO}{\mathcal{O}}

\newcommand{\kl}[2]{D_{
\mathrm{KL}
  }
  \left(#1 \middle| \middle| #2\right)}
\newcommand{\eqdef}{:=}
\newcommand{\bsigma}{\boldsymbol{\sigma}}

\newcommand{\dc}[1]{\todo[color=red!40, inline]{DC: #1}}

\newcommand{\gw}[1]{\todo[color=blue!40, inline]{GW: #1}}

\usepackage{amsmath,graphics}
\newcommand{\decr}[1]{{#1}^{\downarrow}}

\newcommand{\beq}{\begin{eqnarray*}}
\newcommand{\eeq}{\end{eqnarray*}}
\newcommand{\beqn}{\begin{eqnarray}}
\newcommand{\eeqn}{\end{eqnarray}}

\newcommand{\QED}{\hfill\ensuremath{\square}}

\title{Learning discrete distributions with infinite support}
\author{Doron Cohen, Aryeh Kontorovich, Geoffrey Wolfer  \\\small{\texttt{doronv@post.bgu.ac.il, karyeh@cs.bgu.ac.il, geoffrey@post.bgu.ac.il}}}

\begin{document}
\maketitle

\begin{abstract}
  We present a novel approach to estimating discrete distributions with (potentially) infinite support
in the total variation metric.
In a departure from the established paradigm, we make no structural assumptions whatsoever on the sampling
distribution.
In such a setting,
distribution-free risk bounds
are impossible, and the best one could hope for is a fully empirical data-dependent bound.
We derive precisely such bounds, and demonstrate that these are, in a well-defined sense, the best possible.
Our main discovery is that the half-norm
of the empirical distribution
provides tight upper and lower estimates on
the empirical risk.
Furthermore,
this quantity
decays at a nearly optimal rate as a function of the true distribution.
The optimality follows from a minimax result, of possible independent interest.
Additional structural results are
provided, including an exact Rademacher complexity calculation and apparently a first connection
between the total variation risk and the missing mass.

\end{abstract}

\section{Introduction}
\label{section:introduction}
Estimating a
discrete distribution
in the
total variation (TV) metric
is a central problem in computer science and statistics
(see, e.g., 
\citet{han-minimax-l1-2015,DBLP:conf/colt/KamathOPS15,DBLP:conf/nips/OrlitskyS15}
and the references therein).
The TV metric, which we use throughout the paper,
is a natural and abundantly motivated choice
\citep{MR1843146}.
For support size $d$,
a sample of size $\bigO(d/\eps^2)$
suffices for
the maximum-likelihood estimator (MLE)
to be $\eps$-close (with constant probability)
to the unknown target distribution.
A matching lower bound is known
\citep{MR1741038},
and has been computed
down to the exact constants \citep{DBLP:conf/colt/KamathOPS15}.

Classic VC theory --- and, in particular, the aforementioned results --- imply that
for infinite support, no distribution-free sample complexity bound is possible.
If $\dist$ is the target distribution and $\empdist$ is its empirical (i.e., MLE) estimate
based on $m$ iid samples, then \citet{berend2013sharp} showed that
\begin{equation}
\label{eq:BK}
\frac14\Lambda_m(\dist)-\frac1{4\sqrt{m}}
\;\le\;
\E{\tv{\dist-\empdist}}
\;\le\;
\Lambda_m(\dist)
,
\qquad m\ge 2,
\end{equation}
where
\begin{equation}
\label{eq:lamdef}
\Lambda_m(\dist)=
                 \sum_{j\in\N:
                 \dist(j)<  1/m
                 }
                 \dist(j)
+\frac1{2\sqrt m}\sum_{j\in\N:
               \dist(j)\ge1/m
               }
               \sqrt{\dist(j)}
.
\end{equation}
The quantity $\Lambda_m(\dist)$ has the advantage of
always
being finite
and of decaying to $0$ as $m\to\infty$.
The bound in (\ref{eq:BK})
suggests that $\Lambda_m(\dist)$, or a closely related measure,
controls the sample complexity
for learning discrete distributions in TV.
Further supporting the foregoing intuition is
the observation that
for finite support size $d$ and $m\gg1$, we have $\Lambda_m\lesssim\sqrt{d/m}$, recovering the known minimax rate.
Additionally, a closely related measure turns out to
control a minimax risk rate
in a sense made precise in 
Theorem~\ref{theorem:minimax-brief}.

One shortcoming of (\ref{eq:BK})
is that the lower bound only holds for the MLE, leaving the possibility that a different estimator
could achieve significantly improved bounds.
Another shortcoming of (\ref{eq:BK}) and related estimates is that they are not {\em empirical}, in that
they depend on the unknown quantity we are trying to estimate. A fully empirical bound, on the other hand, would give
a high-probability estimate on $\tv{\dist-\empdist}$ solely in terms of observable quantities such as $\empdist$.
Of course, such a bound should also be non-trivial, in the sense of
improving with growing sample size
and approaching $0$ as $m\to\infty$.
A further desideratum might be something akin to {\em instance optimality}: We would like the rate at which the empirical
bound decays to be ``the best'' possible for the given $\dist$, in an appropriate sense. Our analogue of
instance optimality is inspired by, but distinct from, that of \citet{valiant2016instance}, as discussed
in detail in Related work below.

\paragraph{Our contributions.}
We address
the shortcomings of
existing estimators 
detailed above
by providing a fully empirical
bound on
$\tv{\dist-\empdist}$.
Our main discovery is that the quantity
$\Phi_m(\empdist)\eqdef\frac1{\sqrt m}\sum_{j\in\N}\sqrt{\empdist(j)}$
satisfies all of the desiderata posed above for an empirical bound.
As we show in
Theorems~\ref{thm:emp-UB} and \ref{thm:emp-LB},
$\Phi_m(\empdist)$ provides tight, high-probability upper and lower bounds on $\tv{\dist-\empdist}$.
Further,
Theorem~\ref{thm:Phi-exp} shows that
$\E{\Phi_m(\empdist)}$
behaves as $\Lambda_m(\dist)$ defined in (\ref{eq:lamdef}).
Finally, a result in the spirit of instance optimality, Theorem~\ref{thm:Phi-opt},
shows that no other estimator-bound pair can improve upon $(\empdist,\Phi_m)$, other than by small constants.
The latter follows from a minimax bound of independent interest, Theorem~\ref{theorem:minimax-brief}.
Additional structural results are
provided, including an exact Rademacher complexity calculation and a connection (apparently the first)
between the total variation risk and the missing mass.

\paragraph{Definitions, notation and setting.}

As we are dealing with discrete distributions, there is no loss of generality in
taking our sample space
to be the natural numbers
$\N=\set{1, 2, 3, \dots}$.
For $k\in\N$, we write $[k]\eqdef\set{
i\in\N:i\le k
}$.
The set of all distributions
on $\N$
will be denoted by $\Delta_\N$,
which we enlarge to include
the
``deficient'' distributions:
$$
\Delta_\N
\subset
\Delta_\N^\circ
:=
\set{\dist\in[0,1]^\N: \sum_{i\in\N}\dist(i)\le1}
.$$
For $d\in\N$, we write $\Delta_d\subset\Delta_\N$
to denote those $\dist$ whose support is contained in $[d]$.

For $\dist\in\Delta_\N^\circ$ and $I\subseteq\N$,
we write $\dist(I)=\sum_{i\in I}\dist(i)$.
We define the {\em decreasing permutation}
of $
\dist\in\Delta_\N^\circ
$, denoted by $\decr{\dist}$,
to be the sequence $(\dist(i))_{i\in\N}$ sorted in non-increasing order,
achieved by a\footnote{
  While $\decr{\dist}$ is uniquely defined, $\decr{\Pi}_{\dist}$ is not. Uniqueness
  could be ensured by taking the lexicographically first permutation,
  but will not be needed for our results.
}
permutation $\decr{\Pi}_{\dist}:\N\to\N$;
thus, $\decr{\dist}(i)=\dist(\decr{\Pi}_{\dist}(i))$.
For 
$0 < \eta < 1$,
define $T_{\dist}(\eta)\in\N$
as the least $t$ for which
$\sum_{i >t}^{\infty} \decr{\dist}(i) < \eta$.
This induces a truncation of $\dist$, denoted by
$\disttr{\eta}\in\Delta_\N^\circ$
and defined by $\disttr{\eta}(i)=
\pred{\decr{\Pi}_{\dist}(i) \le T_{\dist}(\eta)}
\dist(i)
$.

For $\dist,\bnu\in\Delta_\N^\circ$, 
we define the \emph{total variation distance} in terms of the $\ell_1$ norm:
\begin{equation}
\label{eq:total-variation}
\tv{\dist -\bnu} \eqdef \frac{1}{2} \nrm{\dist -\bnu}_1
=\frac12\sum_{i\in\N}\abs{\dist(i)-\bnu(i)}
.
\end{equation}
For $\dist \in \Delta_\N^\circ$, we also define the \emph{half-norm}\footnote{
  The half-norm
is not a
proper
vector-space norm, as it lacks sub-additivity. 
} as
\begin{equation}
\label{eq:half-norm}
\hfnrm{\dist} \eqdef \left(\sum_{i \in \Omega} \sqrt{\dist(i)} \right)^2;
\end{equation}
note that
while
$\hfnrm{\dist}$
may be infinite,
we have
$\hfnrm{\dist}\le\nrm{\dist}_0$,
where
the latter denotes the support size.

For $m \in \N$ and $\dist \in \Delta_\N$,
we write
$\X = (X_1, \dots, X_m) \sim \dist^{m}$
to mean that the components of the vector $\X$
are drawn iid from from $\dist$.
We reserve $\empdist \in \Delta_\N$
for
the empirical measure induced
by
the sample $\X$, i.e.
$
\empdist(i) \eqdef \frac{1}{m}\sum_{t \in [m]}\pred{X_t = i}$;
the term MLE will be used interchangeably.

For the class of boolean functions over the integers $\set{f \colon \N \to \set{0, 1}}$, 
which we denote by $\{0,1\}^\N$,
recall the definition of
the 
\emph{empirical Rademacher complexity} \citep[Definition~3.1]{mohri2012ml} conditional on the sample $\X$:
\begin{equation}
\label{eq:empirical-rademacher-complexity}
\emprad_m(\X) \eqdef \E[\boldsymbol{\sigma}]{\sup_{f\in \{0,1\}^\N}\frac{1}{m} \sum_{t=1}^m \sigma_t f(X_t)},
\end{equation}
where $\boldsymbol{\sigma} = (\sigma_1, \dots, \sigma_m)
\sim\Unif(\set{-1,1}^m)
$.
The expectation of the above random quantity is the
\emph{Rademacher complexity} \citep[Definition~3.2]{mohri2012ml}:
\begin{equation}
\label{eq:rademacher-complexity}
\rad_m \eqdef \E[\X \sim \dist^{m}]{\emprad(\X)}.
\end{equation}

\paragraph{Related work.}
Given the classical nature of the problem, a comprehensive literature survey is beyond our scope;
the standard texts \citet{MR780746,MR1843146} provide much of the requisite background.
Chapter 6.5 of the latter makes a compelling case for the TV metric used in this paper,
but see \citet{DBLP:conf/innovations/Waggoner15} and the works cited therein for results on other $\ell_p$ norms.
Though surveying all of the relevant literature is a formidable task,
a relatively streamlined narrative may be distilled.
Conceptually, the simplest case is that of $\nrm{\dist}_0<\infty$ (i.e., finite support).
Since learning a distribution over $[d]$ in TV is equivalent to agnostically learning the function class
$\set{0,1}^d$,
standard VC theory
\citep{MR1741038,KonPin2019}
entails
that the MLE achieves the minimax risk rate of $\sqrt{d/m}$ over all $\dist\in\Delta_\N$ with
$\nrm{\dist}_0\le d$.
An immediate consequence is
that in order to obtain
quantitative risk rates for the case of infinite support,
one must assume some sort of structure \citep{DBLP:books/crc/p/Diakonikolas16}.
One can, for example, obtain minimax rates for $\dist$ with bounded entropy
\citep{han-minimax-l1-2015}, or, say, bounded half-norm (as we do here).
Alternatively, one can restrict one's attention to a finite class $\mathcal{Q}\subset\Delta_\N$;
here too, optimal results are known \citep{DBLP:conf/colt/BousquetKM19}.
\citet{berend2013sharp} was one of the few works that made no assumptions
on $\dist\in\Delta_\N$, but only gave non-empirical bounds.

Our work departs from the paradigm of a-priori constraints on the unknown sampling distribution.
Instead, our estimates hold for all $\dist\in\Delta_\N$. Of course, this must come at a price:
no a-priori sample complexity bounds are possible in this setting.
Absent any prior knowledge regarding $\dist$,
one can only hope for sample-dependent
{\em empirical} bounds, and we indeed obtain these. Further, our empirical bounds are essentially the best possible,
as formalized in Theorem~\ref{thm:Phi-opt}. The latter result may be thought of as a learning-theoretic analogue of being
{\em instance-optimal}, as introduced by \citet{valiant2017automatic} in the testing framework.
Instance optimality is
a very natural notion in the context of testing whether an unknown sampling distribution $\dist$ is identical to or $\eps$-far
from a given reference one, $\dist_0$.
For example, \citeauthor{valiant2017automatic} discovered that a truncated $2/3$-norm of $\dist_0$ ---
i.e., a quantity closely related to
$\nrm{\dist_0}_{2/3}$ --- controls the complexity of the testing problem in TV distance.
Instance optimality is more difficult to formalize for distribution learning, since for any given $\dist\in\Delta_\N$,
there is a trivial ``learner'' with $\dist$ hard-coded inside.
\citet{valiant2016instance} defined this notion in terms of competing against an oracle who knows the distribution
up to a permutation of the atoms, and did not provide empirical confidence intervals.
We do derive fully empirical bounds, and further show that
they
are impossible
to improve upon --- by {\em any estimator} --- other than by constants.
Our results suggest that
the half-norm $\hfnrm{\dist}$
plays a role in learning analogous to that of $\nrm{\dist}_{2/3}$ in testing. As an intriguing aside,
we note that
the half-norm
corresponds to the Tsallis $q$-entropy with $q=1/2$, which was shown to be an optimal
regularizer in some stochastic and adversarial bandit settings \citep{DBLP:conf/aistats/ZimmertS19}. We leave the question
of investigating a deeper connection between the two results for future work.

\section{Main results}
\label{section:main-res}
In this section, we formally state our main results.
Recall from the Definitions
that the sample $\X = (X_1, \dots, X_m) \sim \dist^{m}$ induces the
empirical measure (MLE) $\empdist$,
and that a 
key quantity in our bounds is
\begin{equation}
  \label{eq:Phidef}
  \Phi_m(\empdist) \;=\;
\frac{1}{\sqrt{m}}\hfnrm{\empdist}^{1/2}
\;=\;
\frac1{\sqrt m}\sum_{j\in\N}\sqrt{\empdist(j)}.
\end{equation}  

Our first result is a fully empirical, high-probability upper bound on $\tv{\empdist - \dist}$ in terms
of $\Phi_m(\empdist)$:
\begin{theorem}
\label{thm:emp-UB}
For all $m\in\N$, $\delta\in(0,1)$, and $\dist\in\Delta_\N$, we have that
\beq
\tv{\empdist - \dist}
&\leq&
\Phi_m(\empdist)
+3\sqrt{\frac{\log\frac{2}{\delta}}{2m}}
\eeq
holds with probability at least $1-\delta$.
We also have
\beq
\E{\tv{\empdist - \dist}} &\le&  \E{\Phi_m(\empdist)}.
\eeq
\end{theorem}
Since $
\hfnrm{\empdist}
\le
\nrm{\empdist}_0\le\nrm{\dist}_0$, this recovers the minimax rate of $\sqrt{d/m}$ for $\dist\in\Delta_\N$
with $\nrm{\dist}_0\le d$.
We also provide a matching lower bound:
\begin{theorem}
\label{thm:emp-LB}
For all $m\in\N$, $\delta\in(0,1)$, and $\dist\in\Delta_\N$, we have that
\beq
\tv{\empdist - \dist}
&\ge&
\frac1{4\sqrt2}
\Phi_m(\empdist)
- 3\sqrt{\frac{\log\frac{2}{\delta}}{m}}
\eeq
holds with probability at least $1-\delta$.
\end{theorem}

Our empirical measure $\Phi_m(\empdist)$ is never much worse than
the non-empirical
$\Lambda_m(\dist)$, defined in (\ref{eq:lamdef}):
\begin{theorem}
\label{thm:Phi-exp}
For all $m\in\N$ and $\dist\in\Delta_\N$ we have
\beq
\E{\Phi_m(\empdist)}
&\le&
2\Lambda_m(\dist)
\eeq
and,
with probability at least $1-\delta$,
\beq
\Phi_m(\empdist)
&\leq&
2\Lambda_m(\dist) + 
\sqrt{\log({1}/{\delta})/{m}}
.
\eeq

\end{theorem}

Furthermore, no other estimator-bound pair $(\altest,\Psi_m)$ can improve upon
$(\empdist,\Phi_m)$, other than by a constant.
This is the ``instance optimality'' result
alluded to above:
\begin{theorem}
\label{thm:Phi-opt}
There exist universal constants $a, b > 0$ such that the following holds.
For any
estimator-bound
pair $(\altest,\Psi_m)$ 
and any \emph{continuous} %
function $\theta \colon \R_+ \to \R_+$ such that
\beq
\E{\tv{\altest - \dist}}
\;\leq\;
\E{\Psi_m(\altest)}
\;\leq\; \theta \left( \E{\Phi_m(\empdist)} \right)
\eeq
holds
for all $\dist \in \Delta_\N$,
$\theta$ necessarily verifies
\beq
\inf_{0< x < b}  \frac{\theta(x)}{x} \geq \frac{1}{a} %
.
\eeq
\end{theorem}

The next result, framed in the high-probability setting, draws a direct parallel between
our characterization of the learning sample complexity via
the half-norm
and
\citet{valiant2017automatic}'s
characterization of the testing sample complexity
via the $2/3$-norm.
The truncation is needed to ensure
finiteness, since the $\hfnrm{\dist}=\infty$ for heavy-tailed distributions (e.g. $\dist(i) \propto 1/i^2$).

  \begin{theorem}
  \label{theorem:minimax-brief}
  There is a
universal constant  
  $C>0$ such that
for all $\Lambda \geq 2$ and 
$0 < \eps, \delta < 1$,
the MLE
$\empdist
$
verifies the following optimality property:
For all $\dist \in \Delta_\N$
with
$\hfnrm{\disttr{2 \eps \delta /9}} \leq \Lambda$,
we have
\beq
m \geq {C}\eps^{-2} \max \set{ \Lambda, \log(1/\delta) }
\implies
\PR{\tv{\empdist - \dist} < \eps}\ge1 - \delta.
\eeq

On the other hand,
for
{\em any} estimator
$\altest \colon \N^m \to \Delta_\N$
there is a
$\dist \in \Delta_\N$ 
with
$$
\max\set{
  \hfnrm{\disttr{\eps / 18}}
  ,
  \hfnrm{\disttr{2 \eps \delta / 9}}
}
\leq \Lambda$$
such that:
\beq
m <
{C} \eps^{-2}\min \set{ \Lambda, \log(1/\delta) }
\implies
\PR{
  \tv{\altest - \dist} \ge \eps}\ge \min\set{3/4,1-\delta}
.
\eeq

\end{theorem}
The above is a simplified statement chosen for brevity;
a
considerably refined
version is stated and proved in Theorem~\ref{theorem:minimax-long}.

\section{Proofs}
\label{section:proofs}
\subsection{Proof of Theorem~\ref{thm:emp-UB}}
The proof consists of two parts. The first is contained in Lemma~\ref{lemma:tv-empirical-rademacher-upper},
which provides a high-probability empirical upper bound, and an expectation bound, similar to Theorem~\ref{thm:emp-UB}, but in terms of $\emprad_m(\X)$ instead of $\Phi_m(\empdist)$.
The second part, contained in Lemma~\ref{lemma:empirical-rademacher-estimates}, provides an estimate of $\emprad_m(\X)$ in terms of $\Phi_m(\empdist)$.
\begin{lemma}
\label{lemma:tv-empirical-rademacher-upper}
	For all $m\in\N$, $\delta\in(0,1)$, and $\dist\in\Delta_\N$, we have that $$\tv{\empdist - \dist} \leq 2\emprad_m(\X) +3\sqrt{\frac{\log\frac{2}{\delta}}{2m}}$$ holds with probability at least $1-\delta$. We also have,
	\begin{align}
	\label{eq:exp-tv-rad}
		\E{\tv{\empdist - \dist}} \leq  2 \rad_m .
	\end{align}
\end{lemma}

\begin{proof}
The high-probability bound  from the observation,
\begin{align}
\label{eq:tv-uc}
\tv{\empdist - \dist} \eqdef \sup_{A \subseteq \N} \left(\dist(A) - \empdist(A) \right) = 
\sup_{f\in \mathcal{F}} \left(\E[X\sim \dist]{f(X)}  - \frac{1}{m}\sum_{i=1}^{m}f(X_i) \right)
\end{align}
where $\mathcal{F} \eqdef \left\{\mathbb{I}_A | A \subseteq \N \right\} = \{0,1\}^\N$ , combined with
\cite[Theorem~3.3]{mohri2012ml}, which states: 
Let $\mathcal{G}$ be a family of functions from $\mathcal{Z}$ to $[0,1]$ and let $\bnu$ be a distribution supported on a subset of $\mathcal{Z}$. Then, for any $\delta > 0$ , with probability at least $1-\delta$ over $\boldsymbol{Z}=(Z_1, \dots, Z_m) \sim \bnu^{m}$, the following holds:
$$\sup_{g\in \mathcal{G}} \left(\E[Z\sim \bnu]{g(Z)}  - \frac{1}{m}\sum_{i=1}^{m}g(Z_i) \right) \leq 2\emprad_m(\boldsymbol{Z}) +3\sqrt{\frac{\log\frac{2}{\delta}}{2m}}.$$
Plugging in $\mathcal{F}$ for $\mathcal{G}$ and $\dist$ for $\bnu$ in the above theorem completes the proof of the high-probability bound.
The expectation bound (eq. (\ref{eq:exp-tv-rad})) follows from the observation at eq. (\ref{eq:tv-uc}) and a symmetrization argument \cite[eq. (3.8) to (3.13)]{mohri2012ml}.
\end{proof}

In order to complete the proof, we apply
\begin{lemma}[Empirical Rademacher estimates]
\label{lemma:empirical-rademacher-estimates}
	Let $\X=(X_1, \dots, X_m)$ and let $\empdist$ be the empirical measure constructed from the sample $\X$.
	Then, $$\frac{1}{2\sqrt{2}}
        \Phi_m(\empdist) \leq \emprad_m(\X) \leq \frac{1}{2}
        \Phi_m(\empdist).$$
\end{lemma}

\begin{proof} 
The proof is based on an argument that was also developed 
in \citep[Section~7.1, Appendix E.]{scott2006} in the context of histograms and dyadic decision trees, 
and that was credited to Gilles Blanchard.\\
Let $\hat{S}=\left\{X_i|i \in [m]\right\}$ be the empirical support according to the sample $\X=\left(X_1,X_2,...,X_m\right)$.
Then,

\begin{equation*}
\begin{split}
&m\emprad_m (\X) =  \E[\boldsymbol{\sigma}]{\sup_{f \in \{0,1\}^\N} \sum_{i=1}^m \sigma_i f(X_i)} 
  =  \E[\boldsymbol{\sigma}]{ \sup_{A \subseteq \hat{S}} \sum_{i=1}^m \sigma_i \mathbb{I}_A(X_i)}
 \\
  & \qquad =  \sum_{x\in \hat{S}} \E[\boldsymbol{\sigma}]{\sup_{A \subseteq \{x\}} \sum_{i:X_i=x} \sigma_i \mathbb{I}_A(X_i)}
  =  \sum_{x\in \hat{S}} \E[\boldsymbol{\sigma}]{ \left( \sum_{i:X_i=x} \sigma_i \right)_{+}} 
  =  \sum_{x\in \hat{S}} \frac{1}{2} \E[\boldsymbol{\sigma}]{ \left| \sum_{i=1}^{m\empdist(x)} \sigma_i \right| },
\end{split}
\end{equation*}

where the last equality follows from counting $\{i:X_i=x\}$ and the symmetry of the random variable $\sum_{i=1}^{m} \sigma_i$ for all $n\in\N$.
Now, by Khintchine's inequality, for $0<p<\infty$ and $x_1,x_2,...,x_m \in \C$  we have
\beq
\label{eq:khin}
A_p \paren{\sum_{i=1}^m \abs{x_i}^2}^{1/2}
\leq \paren{\E[\boldsymbol{\sigma}]{\abs{\sum_{i=1}^{m} x_i \sigma_i}^p}}^{1/p}
\leq B_p \paren{\sum_{i=1}^m \abs{x_i}^2}^{1/2},
\eeq
where $A_p,B_p > 0$ are constants depending on $p$. Sharp values for $A_p,B_p$ were found by \citet{haagerup1981}. In particular, for $p=1$ he found that $A_1=\frac{1}{\sqrt{2}}$ and $B_1=1$. By using Khintchine's inequality for each $\E[\boldsymbol{\sigma}]{ \left| \sum_{i=1}^{m\empdist(x)} \sigma_i \right| }$ with these constants, we get
$$
\frac{1}{\sqrt{2}} \sqrt{m\empdist(x)}
\leq \E[\boldsymbol{\sigma}]{ \left| \sum_{i=1}^{m\empdist(x)} \sigma_i \right| }
\leq \sqrt{m\empdist(x)}
,
$$
and hence
$$
\frac{1}{2\sqrt{2}} \sum_{x\in \hat{S}}  \sqrt{m\empdist(x)}
\leq m\emprad_m (\X)
\leq \frac{1}{2} \sum_{x\in \hat{S}} \sqrt{m\empdist(x)}
.
$$

Dividing by $m$ completes the proof.
\end{proof}

\noindent Remark: We also give an exact expression for $\emprad_m (\X)$ in Lemma~\ref{lemma:empirical-rademacher-exact}, 
and show in Corollary~\ref{corollary:empirical-rademacher-first-order} with a more delicate analysis that 

$$\frac{\hfnrm{\empdist}^{1/2}}{\sqrt{2\pi m}} - \frac{3}{2}\sqrt{\frac{1}{2\pi }} \frac{1}{m^{3/2}} \nrm{\empdist^+}_{-1/2}^{-1/2} \leq \emprad_m(\X) \leq \frac{\hfnrm{\empdist}^{1/2}}{\sqrt{2\pi m}} + \sqrt{\frac{1}{2\pi }} \frac{1}{m^{3/2}} \nrm{\empdist^+}_{-1/2}^{-1/2}.$$

\subsection{Proof of Theorem~\ref{thm:emp-LB}}
The proof follows from applying the lower bound of Lemma~\ref{lemma:empirical-rademacher-estimates} to the following lemma:

\begin{lemma}[lower bound by empirical Rademacher]
\label{lemma:tv-empirical-rademacher-lower}
	For all $m\in\N$, $\delta\in(0,1)$, and $\dist\in\Delta_\N$, we have that $$\tv{\empdist - \dist} \geq \frac{1}{2}\emprad_m(\X) - 3\sqrt{\frac{\log\frac{2}{\delta}}{m}}$$ holds with probability at least $1-\delta$.
\end{lemma}
\begin{proof}
The proof is closely based on \cite[Proposition~4.12]{wainwright2019hds}, which states:
Let $\boldsymbol{Y}=(Y_1, \dots , Y_m) \sim \bnu^m$ for some distribution $\bnu$ on $\mathcal{Z}$, let $\mathcal{G}\subseteq [-b,b]^\mathcal{Z}$ be a function class, and let $\boldsymbol{\sigma} = (\sigma_1, \dots, \sigma_m)
\sim\Unif(\set{-1,1}^m)$.
Then
\beqn
\label{eq:wrpro4.12}
\sup_{g\in \mathcal{G}} \abs{\E[Y\sim \bnu]{g(Y)}  - \frac{1}{m}\sum_{i=1}^{m}g(Y_i)}
\geq \frac{1}{2}\E[\boldsymbol{\sigma},\boldsymbol{Y}] {\sup_{g\in \mathcal{G}} \abs{\frac{1}{m}\sum_{i=1}^{m} \sigma_i g(Y_i)}} 
- \frac{\sup_{g\in \mathcal{G}}\abs{\E[Y \sim \bnu]{g(Y)}} }{2 \sqrt{m}} 
- \delta
\eeqn
holds with probability at least $1-e^{-\frac{n\delta^2}{2b^2}}$.
Plugging in $\X$ for $\boldsymbol{Y}$, $\dist$ for $\bnu$, $\N$ for $\mathcal{Z}$, $1$ for $b$, and $\mathcal{F} \eqdef \left\{\mathbb{I}_A | A \subseteq \N \right\} = \{0,1\}^\N$ for $\mathcal{G}$ in (\ref{eq:wrpro4.12}) together with observing that
\begin{equation*}
\begin{split}
\tv{\empdist - \dist} \eqdef \sup_{A \subseteq \N} \left(\dist(A) - \empdist(A) \right) = 
\sup_{f\in \mathcal{F}} \abs{\E[X\sim \dist]{f(X)}  - \frac{1}{m}\sum_{i=1}^{m}f(X_i)}
,
\\
\E[\boldsymbol{\sigma},\X]{\sup_{f\in \mathcal{F}}\abs{\frac{1}{m} \sum_{i=1}^m \sigma_i f(X_i)}}
\geq \rad_m, \quad \text{ and } \quad
\sup_{f\in \mathcal{F}}\abs{\E[X \sim \dist]{f(X)}} = 1, 
\end{split}
\end{equation*}
followed by some algebraic manipulation we get
\beqn
\label{eq:lowertvrad}
\tv{\empdist - \dist}
\geq \frac{1}{2}\rad_m
- \frac{1}{2 \sqrt{m}} 
- \sqrt{\frac{2 \log \frac{2}{\delta}}{m}}
\eeqn
with probability at least $1- \delta/2$.
Applying McDiarmid's inequality to the 
$1/m$-bounded-differences function
$\emprad_m(\X)$
(similar to \cite[Eq. (3.14)]{mohri2012ml})
we get:
\beqn
\label{eq:lowerrademprad}
\frac{1}{2}\rad_m \geq 
\frac{1}{2}\emprad_m(\X)
- \frac{1}{2}\sqrt{\frac{ \log \frac{2}{\delta}}{2m}}
\eeqn
with probability at least $1- \delta/2$.
To conclude the proof, combine (\ref{eq:lowertvrad}) and (\ref{eq:lowerrademprad}) with the union bound to get:
\beq
\label{eq:lowertvemprad}
\tv{\empdist - \dist} \geq 
\frac{1}{2}\emprad_m(\X)
- \frac{1}{2 \sqrt{m}} 
- \frac{1}{2}\sqrt{\frac{ \log \frac{2}{\delta}}{2m}}
- \sqrt{\frac{2 \log \frac{2}{\delta}}{m}}
\eeq
with probability at least $1-\delta$, and use the fact $ - \frac{1}{2 \sqrt{m}} -\frac{1}{2}\sqrt{\frac{ \log \frac{2}{\delta}}{2m}}
- \sqrt{\frac{2 \log \frac{2}{\delta}}{m}} \geq -3\sqrt{\frac{\log\frac{2}{\delta}}{m}}$ for all $ m\in \N, \delta\in \paren{0,1}$
.
\end{proof}

\begin{remark}
  We note that by using a more careful analysis, the constants of Theorem~\ref{thm:emp-LB} can be improved
  to yield, 
  under the same assumptions, $\tv{\empdist - \dist} \geq  \frac{1}{2}
  \emprad_m(\X) -\frac{1}{4\sqrt{m}} - \frac{3}{2}\sqrt{\frac{\log\frac{2}{\delta}}{2m}}$
  with probability at least $1-\delta$.
\end{remark}

\subsection{Proof of Theorem~\ref{thm:Phi-exp}}

Invoking Fubini's theorem, we write
$$\frac{1}{\sqrt{m}} \E{\hfnrm{\empdist}^{1/2}} = \frac{1}{m} \sum_{i=1}^{\infty} \E[X \sim \Bin(m,\dist(i))]{ \sqrt{X}}.$$
Since $X\in\set{0,1,2,\ldots}$, we have $\sqrt X\le X$ and hence 
$\E{ \sqrt{X}} \leq \E{ X}$.
On the other hand,
Jensen's inequality
implies
$\E{ \sqrt{X}} \leq \sqrt{\E{X}}$,
whence
\beqn
\label{eq:hfnrm-lambda}
\frac{1}{\sqrt{m}} \E{\hfnrm{\empdist}^{1/2}} &\leq& \frac{1}{m}  \sum_{i=1}^{\infty} \min\{\sqrt{m{\dist}(i)},m{\dist}(i)\} \\
&=& \sum_{i:\ {\dist}(i) \leq  1/m} {\dist}(i) + \frac{1}{\sqrt{m}}  \sum_{i:\ {\dist}(i) > 1/m} \sqrt{{\dist}(i)} \;\leq\; 2\Lambda_m(\dist)
.
\eeqn
The high-probability bound follows from applying McDiarmid's inequality to the 
$2/m$-bounded-differences function: for all $\delta \in \paren{0,1}$,
we have
$$
\Phi_m(\empdist)
\leq
\E{\Phi_m(\empdist)} + 
\sqrt{\log({1}/{\delta})/{m}}.
$$
\QED

\subsection{Statement and proof of the refined version of Theorem~\ref{theorem:minimax-brief}}

\begin{theorem}
\label{theorem:minimax-long}
There is a universal constant $C>0$ such that
for all $\Lambda \geq 2$ and $0 < \eps, \delta < 1$,
the MLE verifies the following optimality property:
For all $\dist \in \Delta_\N$ with
$\nrm{\disttr{2 \eps \delta /9}}_{1/2} \leq \Lambda$,
if $(X_1, \dots, X_m) \sim \dist^{m}$ 
and $m \geq \frac{C}{\eps^2} \max \set{ \Lambda, \ln \delta^{-1} }$,
then
$\tv{\empdist - \dist} < \eps$ holds with probability at least $1 - \delta$.

On the other hand, for all $\Lambda \geq 2$ and $0 < \eps < 1/16, 0 < \delta < 1$,
for {\em any} estimator
$\bar\bmu \colon \N^m \to \Delta_\N$
there is a
$\dist \in \Delta_\N$ 
with $\nrm{\disttr{\eps / 18}}_{1/2} \leq \Lambda$
such that
$\bar\bmu$
must require at least $m \geq \frac{C}{\eps^2} \Lambda$ samples
in order
for $\tv{\bar\bmu - \dist} < \eps$ to hold with probability at least $3/4$,
and for {\em any} estimator $\bar{\boldsymbol{\nu}} \colon \N^m \to \Delta_\N$
there is a $\boldsymbol{\nu} \in \Delta_\N$
with $\nrm{\boldsymbol{\nu}[2 \eps \delta / 9]}_{1/2} \leq \Lambda$,
such that $\bar{\boldsymbol{\nu}}$ must require at least $m \geq \frac{C}{\eps^2} \ln{\frac{1}{\delta}}$
samples in order for $\tv{\bar{\boldsymbol{\nu}} - \boldsymbol{\nu}} < \eps$ to hold with probability at least $1 - \delta$.
\end{theorem}

\paragraph{Minimax risk.}
For any $\Lambda \in [2, \infty), 0 < \eps, \delta < 1$, we define the minimax risk
\begin{equation*}
\label{equation:minimax-risk}
  \risk_m(\Lambda, \eps, \delta) \eqdef \inf_{\bar{\dist}} \sup_{\dist :
    \hfnrm{\disttr{2\eps \delta/9}} < \Lambda} \PR[\X \sim \dist^{m}]{\tv{\bar{\dist} - \dist} > \eps},
\end{equation*}
where the infimum is taken over all functions $\bar{\dist}: \N^m \to \Delta_\N$,
and the supremum is taken over the subset of distributions such that
$\hfnrm{\disttr{2\eps \delta/9}} < \Lambda$.

\paragraph{Upper bound.}

Let $\Lambda \in [2, \infty), 0 < \eps, \delta < 1$, 
$\dist \in \Delta_\N$, such that
$\hfnrm{\disttr{2 \eps \delta /9}} \leq \Lambda$,
$m \in \N$, $(X_1, \dots, X_m) \sim \dist$ and let $\empdist$
be the MLE.
For $\eta>0$, consider the two truncated distributions
$\disttr{\eta}$
and
$\empdisttr$,
where we define the latter as
\beq
\empdisttr(i) &\eqdef& \empdist(i) \pred{\disttr{\eta}(i) > 0 }, \qquad i \in \N.
\eeq

By the triangle inequality, $\PR{\tv{\empdist - \dist} > \eps} \leq \PR{ \calE_1 + \calE_2 + \calE_3 > \eps}$,
where
\begin{equation*}
\begin{split}
\calE_1 \eqdef \tv{\empdist - \empdisttr}, \; \calE_2 \eqdef \tv{\empdisttr - \disttr{2 \eps \delta /9}}, \; \calE_3 \eqdef \tv{\disttr{2 \eps \delta /9} - \dist}.
\end{split}
\end{equation*}
By Markov's inequality,
\begin{equation*}
\begin{split}
  \PR{\calE_1 >
    \frac{\eps}{3}} &\leq \frac{3}{\eps} \E{ \tv{\empdist - \empdisttr}}
  = \frac{3}{2 \eps} \E{\sum_{i=1}^{\infty} \abs{\empdist(i) - \empdisttr(i)}} \\
&= \frac{3}{2 \eps} \E{ \frac{1}{m} \sum_{i \in \N \colon \Pi_{\dist}(i) > T_{\dist}(\eta)} \sum_{t = 1}^{m} \pred{X_t = i}} = \frac{3}{2 \eps}  \PR{\Pi_{\dist}(X_t) > T_{\dist}(\eta)} \leq \frac{\delta}{3}.
\end{split}
\end{equation*}
Moreover, $\calE_3 = \frac{1}{2} \sum_{i > T_{\dist}(\eta)}^{\infty} \decr{\dist}(i) \leq \frac{\eps \delta}{9} \leq \frac{\eps}{3}$.
In order to apply
the union bound, 
it remains to handle $\PR{\calE_2 > \eps /3}$.
This is achieved in two standard steps.
The first follows an argument similar to that of \cite[Lemma~5]{berend2013sharp}, 
that bounds from above the quantity in expectation using Jensen's inequality,
$\E{\calE_2} \leq \frac{\hfnrm{\disttr{2 \eps \delta /9}}^{1/2}}{\sqrt{m}} \leq \sqrt{\frac{\Lambda}{m}}$.
An application of McDiarmid's inequality controls the 
fluctuations around the expectation \citep[(7.5)]{berend2013sharp} and concludes the proof.

\QED

\paragraph{Sample complexity lower bound $m = \Omega \left( \frac{\log \delta^{-1}}{\eps^2} \right)$.}

See Lemma~\ref{lemma:minimax-lower-bound-delta}.

\paragraph{Sample complexity lower bound $m = \Omega \left( \frac{\Lambda}{\eps^2} \right)$.}

Let $\eps \in (0, 1/16)$
and
$\Lambda > 2$. 
First observe that
$\Lambda / 2 \leq 2\floor{\Lambda/2} \leq \Lambda$, and $2\floor{\Lambda/2} \in 2 \N$.
As a result,
\begin{equation*}
\begin{split}
  &\risk_m(\Lambda, \eps, \delta) \stackrel{(i)}{\geq} \inf_{\bar{\dist}} \sup_{\dist :
    \hfnrm{\disttr{2\eps \delta/9}} \leq 2 \floor{\Lambda/2}} \PR[\X \sim \dist^{m}]{\tv{\bar{\dist} - \dist} > \eps} \\
  &\stackrel{(ii)}{\geq} \inf_{\bar{\dist}} \sup_{\dist \in \Delta_{2 \floor{\Lambda/2}}}
     \PR[\X \sim \dist^{m}]{\tv{\bar{\dist} - \dist} > \eps} \stackrel{(iii)}{\geq} \frac{1}{2} \left(1 -  \frac{ m C \eps^2 }{2 \floor{\Lambda/2}} \right)
  \geq \frac{1}{2} \left(1 -  \frac{ 2 m C \eps^2 }{\Lambda} \right) \\
\end{split}
\end{equation*}
where $(i)$ and $(ii)$ follow 
from taking the supremum over increasingly smaller sets, 
$(iii)$ is Lemma~\ref{lemma:minimax-lower-bound-tsybakov} 
invoked for $2 \floor{\Lambda/2} \in \N$,
and $C>0$ is a universal constant.
To conclude, $m \leq \frac{\Lambda}{4C \eps^2} \implies \risk_m(\Lambda, \eps, \delta) \geq 1/4$,
which yields the second lower bound.
\QED

Remark: The universal constant in the lower bound obtained by Tsybakov's method at Lemma~\ref{lemma:minimax-lower-bound-tsybakov} 
is suboptimal, and we give a short proof in the appendix for completeness. 
We refer the reader to the more involved methods of \citet{DBLP:conf/colt/KamathOPS15} for obtaining tighter bounds.

\subsection{Proof of Theorem~\ref{thm:Phi-opt}}

Let $d \in 2\N$ and $m \in \N$, and restrict the problem to $\dist \in \Delta_d$.
Let $\eps \in (0, 1/16)$.
By Lemma~\ref{lemma:minimax-lower-bound-tsybakov},
$\riskf_m(d, \eps) \eqdef \inf_{\bar{\dist}} \sup_{\dist \in \Delta_d} \PR{\tv{\bar{\dist} - \dist} > \eps} \geq \frac{1}{2} \left(1 -  \frac{ C m \eps^2 }{d} \right)$ for some $C>0$,
whence Markov's inequality yields
$$\frac{1}{2} \left(1 -  \frac{ C m \eps^2 }{d} \right) \leq \frac{1}{\eps} \inf_{\bar{\dist}} \sup_{\dist \in \Delta_d} \E{\tv{\bar{\dist} - \dist}}.$$
Restrict $m \geq \frac{d}{b^2}$, with $b \eqdef \sqrt{3C/16}$ and set $\eps = \sqrt{\frac{d}{3C m}}$, so that
\begin{equation}
\label{eq:minimax-lb-finite}
\inf_{\bar{\dist}} \sup_{\dist \in \Delta_d} \E{\tv{\bar{\dist} - \dist}} \geq \frac{1}{a}\sqrt{\frac{d}{m}}, \text{ where } a \eqdef \sqrt{27C}
\end{equation}

Suppose that $\theta(\sqrt{d/m}) < \frac{1}{a}\sqrt{\frac{d}{m}}$, then by hypothesis,
$$\inf_{\bar{\dist}} \sup_{\dist \in \Delta_d} \E{\tv{\bar{\dist} - \dist}} 
\leq \sup_{\dist \in \Delta_d} \E{\Psi_m(\altest)} 
\leq  \sup_{\dist \in \Delta_d} \theta \left(\E{\Phi_m} \right).$$
For $\dist \in \Delta_d$, $\E{\sqrt{\frac{\hfnrm{\empdist}}{m}}} \leq \sqrt{\frac{d}{m}}$.
It follows that
$$\sup_{\dist \in \Delta_d} \theta \left(\E{\Phi_m}\right) \leq \theta\left( \sqrt{\frac{d}{m}} \right) <  \frac{1}{a}\sqrt{\frac{d}{m}},$$
which
contradicts
\eqref{eq:minimax-lb-finite}.
We have therefore
established,
for $$r \in R \eqdef \set{\sqrt{d/m} \colon (m, d) \in \N \times 2\N, m \geq \frac{d}{b^2}},$$
the lower bound
$\theta(r) \geq r/a$. 
We extend the lower bound to the open interval $(0, b)$,
by observing that $R$ is dense in $(0, b)$ followed by a continuity argument. 
\QED

\section*{Acknowledgments}
We are thankful to Clayton Scott for the insightful conversations, 
and to the anonymous referees for their valuable comments.
This research was partially supported by
the Israel Science Foundation
(grant No. 1602/19) and Google Research.

\bibliography{bibliography}
\bibliographystyle{abbrvnat}

\appendix
\section{Analysis of the Empirical Rademacher complexity}
From Lemma~\ref{lemma:empirical-rademacher-estimates} (see also \citep[Section~7.1, Appendix E.]{scott2006}), 
we see that the Khintchine inequality already yields a control of $\emprad_m(\X)$ by $\hfnrm{\empdist}^{1/2}$ 
up to universal constants.
$$\frac{1}{2\sqrt{2}} \hfnrm{\empdist}^{1/2} \leq \emprad_m(\X) \leq  \frac{1}{2}\hfnrm{\empdist}^{1/2}.$$
Furthermore, it is possible to derive an exact expression for it, 
from the expected absolute distance of a symmetric random walk:
\begin{lemma}[Empirical Rademacher complexity, exact expression]
\label{lemma:empirical-rademacher-exact}
	Let $\X=(X_1, \dots, X_m)$ and let $\empdist$ be the empirical measure constructed from the sample $\X$.
	Then, $$\emprad_m(\X) = \frac{1}{m}\sum_{x :\ \empdist(x)>0} \frac{1}{2^{{m\empdist(x)}}}\left\lceil \frac{{m\empdist(x)}}{2}\right\rceil\binom{{m\empdist(x)}}{\lceil {m\empdist(x)}/2\rceil}.$$
\end{lemma}

\begin{proof}
  Write $m\emprad_m (\X) =  \sum_{x :\ \empdist(x)>0} \frac{1}{2} \E[\boldsymbol{\sigma}]{ \left| \sum_{i=1}^{m\empdist(x)} \sigma_i \right| }$
  as in the proof of Lemma~\ref{lemma:empirical-rademacher-estimates}.\\
  Now, observe that $\E[\boldsymbol{\sigma}]{ \left| \sum_{i=1}^{m\empdist(x)} \sigma_i \right| }$ is the expectation value
  of the absolute distance of a 1-dimensional symmetric random walk after $m\empdist(x)$ steps,
also known as the ``heads minus tails'' process \citep{10.2307/2686609}:
$$\E[\boldsymbol{\sigma}]{ \left| \sum_{i=1}^{m\empdist(x)} \sigma_i \right|} = \frac{1}{m 2^{{m\empdist(x)}}}\left\lceil \frac{{m\empdist(x)}}{2}\right\rceil\binom{{m\empdist(x)}}{\lceil {m\empdist(x)}/2\rceil}.$$
\end{proof}

However, the above is inconvenient and involves the computation of factorials.
Leveraging delicate bounds for the central binomial 
coefficient obtained with the Wallis product in \citet{dunbar2009topics}, 
we derive the following corollary, that gives exact the first-order 
constant in terms of the \emph{half-norm}, 
makes the \emph{minus-half-norm} appear as a second dominant term,
and that is easily computable.

\begin{corollary}[Empirical Rademacher complexity, first order bound]
\label{corollary:empirical-rademacher-first-order}
Let $\X=(X_1, \dots, X_m)$ and let $\empdist$ be the empirical measure constructed from the sample $\X$.
Then writing 
$$\phi_m(\X) \eqdef \frac{\hfnrm{\empdist}^{1/2}}{\sqrt{2\pi m}},$$
it holds that
\begin{equation*}
\begin{split}
\frac{\hfnrm{\empdist}^{1/2}}{\sqrt{2\pi m}} - \frac{3}{2}\sqrt{\frac{1}{2\pi }} \frac{1}{m^{3/2}} \nrm{\empdist^+}_{-1/2}^{-1/2} \leq \emprad_m(\X) \leq \frac{\hfnrm{\empdist}^{1/2}}{\sqrt{2\pi m}} + \sqrt{\frac{1}{2\pi }} \frac{1}{m^{3/2}} \nrm{\empdist^+}_{-1/2}^{-1/2}.
\end{split}
\end{equation*}

\end{corollary}
\begin{proof}
Let $n \in \N$, if $n = 2k, k \geq 1$,
\begin{equation*}
\begin{split}
\frac{1}{2^{n}}\left\lceil \frac{n}{2}\right\rceil\binom{n}{\lceil {n}/2\rceil} &= 
\frac{1}{4^k} k \binom{2k}{k}, \\
\end{split}
\end{equation*}
and if $n = 2k - 1, k \geq 1$, $\lceil {n}/2\rceil = k$ such that similarly,
\begin{equation*}
\begin{split}
\frac{1}{2^{n}}\left\lceil \frac{n}{2}\right\rceil\binom{n}{\lceil {n}/2\rceil} &= 
\frac{1}{2^{2k-1}} k \binom{2k-1}{k} = \frac{2}{4^k} k \frac{(2k-1)!}{k!(2k-k-1)!} \\
&= \frac{2}{4^k} k \frac{(2k)!(2k-k)}{(2k)k!(2k-k)!} = \frac{2}{4^k} k \frac{2k-k}{2k} \binom{2k}{k} 
= \frac{1}{4^k} k \binom{2k}{k}. \\
\end{split}
\end{equation*}
Moreover, from \citet[p.11]{dunbar2009topics}, for $k \geq 1$, an application of the Wallis product yields,
\begin{equation*}
\begin{split}
\frac{k}{\sqrt{\pi/2}\sqrt{2k + 1}}\left( 1 - \frac{1}{2k} \right) \leq \frac{1}{4^k} k \binom{2k}{k} \leq \frac{k}{\sqrt{\pi/2}\sqrt{2k + 1}}\left( 1 + \frac{1}{2k} \right). \\
\end{split}
\end{equation*}
If follows that when $n = 2k$,
\begin{equation*}
\begin{split}
\sqrt{\frac{n}{2\pi}}\left\{\sqrt{\frac{n}{n + 1}}\left( 1 - \frac{1}{n} \right)\right\} \leq \frac{1}{2^{n}}\left\lceil \frac{n}{2}\right\rceil\binom{n}{\lceil {n}/2\rceil} \leq \sqrt{\frac{n}{2\pi}}\left\{\sqrt{\frac{n}{n + 1}}\left( 1 + \frac{1}{n} \right)\right\}, \\
\end{split}
\end{equation*}
and for $n = 2k -1$,
\begin{equation*}
\begin{split}
\sqrt{\frac{n}{2\pi}}\left\{ \frac{n+1}{\sqrt{n(n+2)}} \left( 1 - \frac{1}{n+1} \right) \right\} \leq \frac{1}{2^{n}}\left\lceil \frac{n}{2}\right\rceil\binom{n}{\lceil {n}/2\rceil} \leq \sqrt{\frac{n}{2\pi}}\left\{ \frac{n+1}{\sqrt{n(n+2)}} \left( 1 + \frac{1}{n+1} \right) \right\} \\
\end{split}
\end{equation*}
For all $n \in \N$,
\begin{equation*}
\begin{split}
\sqrt{\frac{n}{n + 1}}\left( 1 + \frac{1}{n} \right) &\leq  \frac{n+1}{\sqrt{n(n+2)}} \left( 1 + \frac{1}{n+1} \right) \leq 1 + \frac{1}{n}, \\
\sqrt{\frac{n}{n + 1}}\left( 1 - \frac{1}{n} \right) &\geq  \frac{n+1}{\sqrt{n(n+2)}} \left( 1 - \frac{1}{n+1} \right) \geq 1 - \frac{3}{2n},
\end{split}
\end{equation*}
such that
\begin{equation*}
\begin{split}
\emprad_m(\X) &\leq \sqrt{\frac{1}{2\pi m}}\sum_{x \colon \empdist(x)>0} \sqrt{\empdist(x)}\left\{1 + \frac{1}{m \empdist(x)} \right\} \\
&\leq \phi_m(\X) + \sqrt{\frac{1}{2\pi }} \frac{1}{m^{3/2}} \nrm{\empdist^+}_{-1/2}^{-1/2}, \\
\end{split}
\end{equation*}
where we wrote
$$\nrm{\empdist^+}_{-1/2}^{-1/2} \eqdef \sum_{x \in \N} \frac{\pred{\empdist(x) > 0}}{\sqrt{\empdist(x)}},$$
and conversely,
\begin{equation*}
\begin{split}
\emprad_m(\X) &\geq \phi_m(\X) - \frac{3}{2}\sqrt{\frac{1}{2\pi }} \frac{1}{m^{3/2}} \nrm{\empdist^+}_{-1/2}^{-1/2}. \\
\end{split}
\end{equation*}

\end{proof}

\section{Auxiliary lemmas for lower bounds}
\begin{lemma}[Sample complexity lower bound $m = \Omega \left( {\log \delta^{-1}}/{\eps^2} \right)$]
\label{lemma:minimax-lower-bound-delta}
Let $\Lambda \geq 2$, $0 < \eps < 1/2$ and $0 < \delta < 1$.
For {\em any} estimator $\bar{\boldsymbol{\nu}} \colon \N^m \to \Delta_\N$
there is a $\boldsymbol{\nu} \in \Delta_\N$
with $\nrm{\boldsymbol{\nu}[2 \eps \delta / 9]}_{1/2} \leq \Lambda$,
such that $\bar{\boldsymbol{\nu}}$ must require at least $m = \Omega \left( \frac{\log \delta^{-1}}{\eps^2} \right)$
samples in order for $\tv{\bar{\boldsymbol{\nu}} - \boldsymbol{\nu}} < \eps$ to hold with probability at least $1 - \delta$.
\end{lemma}

\begin{proof}
The proof is standard and consists of lower bounding the difficulty of learning a biased coin. 
Recall that for $\dist_0 \eqdef (1/2, 1/2), \dist_\eps \eqdef (1/2 - \eps, 1/2 + \eps)$, 
direct computations lead to $\nrm{\dist_0 - \dist_\eps}_1 = 2\eps$, and 
$\kl{\dist_\eps}{\dist_0} = (1/2 - \eps) \ln \frac{1/2 - \eps}{1/2} + (1/2 + \eps) \ln \frac{1/2 + \eps}{1/2} \leq 4 \eps^2$, 
where $\kl{\dist_\eps}{\dist_0}$ is the KL divergence between $\dist_\eps$ and $\dist_0$.
We also verify that $\hfnrm{\dist_\eps} \leq \hfnrm{\dist_0} \leq 2 \leq \Lambda$, 
hence also for their truncated version.
From an immediate corollary of LeCam's theorem \citep[Theorem~2.2, Lemma~2.6]{tsybakov2009introduction}, 
$\risk_m(\Lambda, \eps, \delta) \geq \frac{1}{2} \expo{ - m\kl{\dist_\eps}{\dist_0}}$,
whence $m \leq \frac{1}{4 \eps^2} \log \frac{\delta^{-1}}{2} \implies \risk_m(\Lambda, \eps, \delta) \geq \delta$.
\end{proof}

\begin{lemma}
\label{lemma:minimax-lower-bound-tsybakov}
Let $d \in 2\N, d \geq 16, m \in \N, \eps \in (0, 1/16)$, and let
\begin{equation*}
\begin{split}
\riskf_m(d, \eps) \eqdef \inf_{\bar{\dist}} \sup_{\dist :
    \dist \in \Delta_d} \PR[\X \sim \dist^{m}]{\tv{\bar{\dist} - \dist} > \eps},
\end{split}
\end{equation*}
where the infimum is taken over all $\bar{\dist} \colon [d]^m \to \Delta_d$. Then
there is a universal $C>0$
such that
\begin{equation*}
\begin{split}
\riskf_m(d, \eps) \geq \frac{1}{2} \left(1 -  \frac{ C m \eps^2 }{d} \right).
\end{split}
\end{equation*}
\end{lemma}
\begin{proof}
As is customary in Analysis, the universal constant $C>0$ may change its value from expression to expression.
Consider the family of distributions 
$$\mathcal{D}(d) \eqdef \set{ \dist^{(\bsigma)} \in \Delta_d, \bsigma \in \set{ 0,1 }^{d / 2}},$$
where 
\begin{equation*}
\dist^{(\bsigma)} \eqdef \frac{1}{d} \left( 1 + 16 \eps \sigma_1, 1 - 16 \eps \sigma_1, 1 + 16 \eps \sigma_2, 1 - 16 \eps \sigma_2, \dots, 1 + 16 \eps \sigma_{d/2}, 1 - 16 \eps \sigma_{d/2} \right).
\end{equation*}
From the Varshamov-Gilbert bound \citep[Lemma~2.9]{tsybakov2009introduction},
there exists a
$\tilde{\mathcal{D}}(d) \subsetneq \mathcal{D}(d)$ 
satisfying
$(a)$
$\abs{\tilde{\mathcal{D}}(d)} > 2^{d/16}$,
$(b)$
for $\dist^{(\bsigma)}, \dist^{(\bsigma')} \in \tilde{\mathcal{D}}(d)$, 
$ \bsigma \neq \bsigma' \implies \tv{\dist^{(\bsigma)} - \dist^{(\bsigma')}} \geq 2 \eps$, 
and
$(c)$
$\dist^{(\zero)}  \in \tilde{\mathcal{D}}(d)$.
It is straightforward to
verify that $\kl{\dist^{(\bsigma)}}{\dist^{(\zero)}} \leq C \eps^2$. 
Applying Tsybakov's method \citep[Theorem~2.5]{tsybakov2009introduction},
\begin{equation*}
\begin{split}
\riskf_m(d, \eps) &\geq \inf_{\bar{\dist}} \sup_{\dist \in \tilde{\mathcal{D}}(d)} \PR{\tv{\bar{\dist} - \dist} > \eps} \\
& \geq \frac{1}{2} \left(1 -  \frac{\frac{4 m}{\abs{\tilde{\mathcal{D}}(d)}} \sum_{\dist^{(\bsigma)} \in \tilde{\mathcal{D}}(d)} \kl{\dist^{(\bsigma)}}{\dist^{(\zero)}}}{\ln \abs{\tilde{\mathcal{D}}(d)}} \right) \\
\end{split}
\end{equation*}
so that $\riskf_m(d, \eps) \geq \frac{1}{2} \left(1 -  \frac{ C m \eps^2 }{d} \right)$.
\end{proof}

\section{Convergence properties of the empirical bound}

In this section, we briefly analyze convergence of $\frac{1}{\sqrt{m}}\hfnrm{\empdist}^{1/2}$.
In Proposition~\ref{proposition:as-convergence} we confirm that the quantity converges
\emph{almost surely} and in \emph{$L_1$}, but with Proposition~\ref{proposition:arbitrary-slow},
with show that this convergence can be \emph{arbitrarily slow}.

\begin{proposition}[$L_1$ and almost sure convergence]
\label{proposition:as-convergence}
	Let $\dist \in \Delta_\N$ and let $\X \eqdef (X_1, \dots, X_m) \sim \dist^{m}$.
	Then, $\frac{1}{\sqrt{m}}
        \hfnrm{\empdist}^{1/2} \xrightarrow{L_1} 0$ and $\frac{1}{\sqrt{m}}
        \hfnrm{\empdist}^{1/2} \xrightarrow{\text{a.s.}} 0$. \\
\end{proposition}
\begin{proof}
For $L_1$ convergence the proof is as follows:
\begin{align*}
    \lim_{m \to \infty} \E{\left|\frac{1}{\sqrt{m}}
      \hfnrm{\empdist}^{1/2} - 0\right|} 
      & =\lim_{m \to \infty} \E{\frac{1}{\sqrt{m}} \hfnrm{\empdist}^{1/2}} \\
    & \leq \lim_{m \to \infty} 2\Lambda_m(\dist) 
    \hspace{50.1mm} \text{(Theorem~\ref{thm:Phi-exp})} \\
    & = 0. \hspace{25.3mm} \text{(\cite[Lemma~7]{berend2013sharp})}
\end{align*}

Now, for almost sure convergence, recall that
$\frac{1}{\sqrt{m}}\hfnrm{\empdist}^{1/2}$
satisfies $2/m$-bounded-differences.
By the $L_1$ convergence established above,
we have that for all $\eps>0$ there is
an
$M_{\eps} \in \N$ s.t.
for all
$m\geq M_{\eps}$,
we have
$\E{\frac{1}{\sqrt{m}}
  \hfnrm{\empdist}^{1/2}} \leq \eps/2$.
Invoking McDiarmid's inequality, for every $m \geq M_{\eps}$, we have
\begin{align*}
  \mathbb{P}\left( \frac{1}{\sqrt{m}}
  \hfnrm{\empdist}^{1/2} \geq \eps/2 \right) \leq \exp\left( -\frac{m\eps^2}{2} \right)
  .
\end{align*}

Thus,
\begin{align*}
  \sum_{m = 1}^{\infty} \mathbb{P}\left( \frac{1}{\sqrt{m}}
  \hfnrm{\empdist}^{1/2} \geq  \eps/2 \right) &
 \leq M_{\eps} + \sum_{m = M_{\eps}}^{\infty} \exp\left( -\frac{m\eps^2}{2} \right) < \infty.
\end{align*}

An application of the Borel-Cantelli lemma completes the proof:
$$
\frac{1}{\sqrt{m}} \hfnrm{\empdist}^{1/2}  \xrightarrow{\text{a.s.}} 0
.
$$

\end{proof}

To formalize our idea of arbitrarily slow convergence,
we adapt the terminology developed in \citet{deutsch2010slow1,deutsch2010slow2}.
We begin with the set of all $[0,1]$-valued sequences 
that converge to $0$:
$$\mathcal{U} \eqdef \{ U\in [0,1]^\N:
\lim_{m \rightarrow \infty} U(m) = 0 \}.$$
Following
\citet[Definition~2.7]{deutsch2010slow1},
we will say that the statistic $\hat{\theta}_m \colon \N^m \to [0, 1]$ \emph{converges arbitrarily slowly to $0$ in $L_1$}
if
\begin{enumerate}
	\item $\forall \mu \in \Delta_\N$, $\lim_{m \rightarrow \infty }\E{\hat{\theta}_m} = 0$,
	\item $\forall U \in \mathcal{U}, \exists \mu \in \Delta_\N$ such that $\forall m \in \N, \E{\hat{\theta}_m} \geq U(m)$.
\end{enumerate}
It turns out
\citep[Remark~2.8, Theorem~2.9]{deutsch2010slow2}
that restricting the set $\mathcal{U}$ to the {\em decreasing} sequences,
$$\mathcal{U}^{\downarrow} \eqdef \{
U\in[0,1]^\N
:
\sup_{m\in\N}U(m+1)/U(m)\le1,
\lim_{m \rightarrow \infty} U(m) = 0 \},$$
does not change the above definition of arbitrarily slow convergence.

\begin{proposition}[Arbitrary slow convergence in $L_1$]
\label{proposition:arbitrary-slow}
  For any sequence $1>r_1>r_2>\dots$ decreasing to $0$, there is a distribution $\dist\in \Delta_\N$ such
  that $2\E[\X\sim \dist^{m}]{\tv{\dist-\empdist}} > r_m$ for all $m\geq1$.
\end{proposition}

\begin{proof}

\begin{align*}
  2\E{ \tv{\dist - \empdist}}
  &=
  \E{ \nrm{ \dist - \empdist }_1} \\
& = \sum_{i=1}^{\infty} \E{ \abs{ \dist(i) - \empdist(i) }}\\
  & = \sum_{i=1}^{\infty}
  \E{ \pred{\empdist(i) > 0}\abs{ \dist(i) - \empdist(i) } +  \pred{\empdist(i) = 0}\abs{ \dist(i) - \empdist(i) }}\\
  & = \sum_{i=1}^{\infty} \E{ \pred{\empdist(i) > 0}\abs{ \dist(i) - \empdist(i) }} +
  \sum_{i=1}^{\infty} \E{\pred{\empdist(i) = 0}\dist(i)}\\
  & = \sum_{i=1}^{\infty} \E{ \pred{\empdist(i) > 0}\abs{ \dist(i) - \empdist(i) }} +
  \E{\sum_{i:\ \empdist(i) = 0} \dist(i)}\\
& \geq \E{\sum_{i:\ \empdist(i) = 0} \dist(i)} = \E{\dist\paren{\N\setminus \{X_1, \dots X_m\}}}\\
& = \E{U_m}
,
\end{align*}
where $U_m \eqdef \dist\paren{\N\setminus \{X_1, \dots X_m\}}$ 
is the missing mass random variable.
From \cite[Proposition~4]{berend2012missingmass}, we have that:
For any sequence $1>r_1>r_2>\dots$ decreasing to $0$, there is a distribution $\dist\in \Delta_\N$ such
  that $\E{U_m} > r_m$ for all $m\geq1$.
\end{proof}
 
\begin{remark}
  To our knowledge,
  the above result
  is the first to establish
  a
  connection between the TV risk
  $\tv{\dist - \empdist}$ and the missing mass $U_m$.
\end{remark}

\end{document}